\documentclass[leqno]{amsart}
\usepackage{amsmath}        
\usepackage{amsfonts}
\usepackage{enumerate}
\usepackage{amsthm}
\usepackage{amssymb}
\usepackage{amscd}
\usepackage[all]{xy}
\usepackage[titletoc,toc,title]{appendix}

\newtheorem{thm}{Theorem}[section] 

\newtheorem{cor}[thm]{Corollary}
\newtheorem{lemma}[thm]{Lemma}

\theoremstyle{definition}

\theoremstyle{definition}

\newtheorem*{rem}{Remark}

\def\sk1{\vskip 10pt}
\def\newline{\hfil\break}

\def\del{\partial}

\def\sk{\vskip}

\numberwithin{equation}{section}

\title{Semistability and $CAT(0)$ geometry}

\author{Ross Geoghegan}

\address{\noindent Ross Geoghegan, Department of Mathematical Sciences,
Binghamton University (SUNY),
Binghamton, NY 13902-6000, USA}
\email{ross@math.binghamton.edu}


\date{March 16, 2017}


\begin{document}
\fontsize{12}{13pt} \selectfont  

\begin{abstract}
We explain why semistability of a one-ended proper $CAT(0)$ space can be determined by the
geodesic rays. This is applied to boundaries of $CAT(0)$ groups.
\end{abstract}

\maketitle

\section{Geometry}

Let $X$ be a proper $CAT(0)$ space having one end, and let $\del X$
denote its compactifying boundary. One says that $X$ has {\it semistable
fundamental group at infinity} (or {\it has one strong end})  if any two proper rays in $X$ are
properly homotopic.  A point of $\del X$ is, by definition, an equivalence class
of geodesic rays in $X$ any two of which are boundedly close. (See \cite
{BrHa99} for $CAT(0)$ matters.) Since geodesic rays are proper one may
ask if semistability can be determined by the geodesic rays alone.

\begin{thm}\label{main} The one-ended proper $CAT(0)$ space $X$ has semistable
fundamental group at infinity if and only if any two geodesic rays in $X$ are properly homotopic.
\end{thm} 

\begin{rem} If it is the case that any two geodesic rays in $X$ are
properly homotopic through geodesic rays (i.e. every level of the homotopy a geodesic ray), then
$\del X$ is path
connected. Theorem \ref{main} describes something weaker than path
connectedness.  \end{rem}

An inverse sequence of groups $\{H_{n}, f^{m}_{n}\}$ is said to be {\it
semistable} (or {\it Mittag-Leffler}) if, for each $n$, the images of
the bonding homomorphisms $f^{m}_{n}:H_{m}\to H_{n}$ are the same for all
but finitely many values of $m>n$. The relevance is this: Let $\{K_{n}\}$
be an exhausting sequence of compact subsets of $X$ such that, for all
$n$, $K_{n}$ is a subset of the interior of $K_{n+1}$. Choose a suitably parametrized proper
base ray $\omega $ in $X$ and consider the inverse sequence of groups
$\{\pi _{1}(X-K_{n}, \omega), \text{(inclusion)}_{\#}\}$. This sequence of groups
is semistable if and only if any two proper rays in $X$ are properly
homotopic; see \cite{G}, Section 16.1 for details. This explains the
terminology above.

\begin{rem} It has long been known (\cite {K}) that if a
metrizable compact connected space (e.g. $\del X$) is locally
connected\footnote{Equivalently,  if it is a Peano space,  i.e. the
continuous image of a closed interval.} then it has a shape theoretic property called ``pointed
1-movable". Our $\del X$ has this property if and only if $X$ has semistable
fundamental group at infinity. So Theorem \ref{main} is mainly of interest
when $\del X$ is not known to be locally connected.  \end{rem}

\begin{rem} That theorem of Krasinkiewicz \cite {K} says
more: A metrizable compact connected space is pointed 1-movable {\it if and
only if it is shape equivalent to} a locally connected compact connected
metrizable space. This applies to our $\del X$.
\end{rem}

\begin{rem} Any compact metrizable space can be the boundary of a proper
$CAT(0)$ space\footnote{For a sketched proof of this observation of
Gromov see the Appendix to \cite {GO07}.}, so there are many proper
$CAT(0)$ spaces $X$ which do not have semistable fundamental group at
infinity. For example, a dyadic solenoid can be such a boundary. 
\vskip 5pt
{\bf Open Question $1$:} {\it Is it true that when the full isometry group of $X$ acts
cocompactly on $X$ then $X$ has semistable fundamental group at infinity?}
\vskip 5pt
(Compare this with the better-known group theoretical Open Question $2$ posed in Section $2$.) 
\vskip 5pt

We note that a necessary condition for the full isometry
group of $X$ to act cocompactly is that the Lebesgue covering dimension
and the cohomological dimension of $\del X$ be the same \cite {GO07}.
\end{rem}

We write $\widehat X$ for the compact absolute retract (AR) $X\cup \del
X$. It is easy to see that the set $\del X$ is a $Z$-set in $\widehat
X$. Recall that this means that for any open set $U$ in $\widehat
X$ the inclusion map $U\cap X\to U$ is a homotopy
equivalence\footnote{Less formally: for any positive $\epsilon $, any map
of a polyhedron $P$ into $\widehat X$ can be $\epsilon $-homotoped off
$\del X$, holding fixed an arbitrarily large closed subset of $P$ lying in the pre-image
of $X$.}. So shape theory can be directly applied to the metric compactum $\del X$ as it sits
naturally in the AR $\widehat X$.

We need some shape theoretic terminology. 
\begin{enumerate}[(1)]
\item A {\it strong shape component} of $\del X$ is a proper homotopy class of proper rays in
$X$. 
\vskip 5pt
\item The proper ray $c:[0,\infty )\to X$ {\it ends at the point} $p\in
\del X$ if the map $c$ extends continuously to ${\hat c}:[0,\infty ]\to
\widehat X$ by mapping the point $\infty $ to $p$. 
\vskip 5pt
\item Two points $p$ and $q$ of $\del X$ lie in the same {\it component of
joinability} if there are proper rays in $X$ ending at $p$ and at $q$
which are properly homotopic.

\end{enumerate} 

\begin{rem} For strong shape theory see, for example, Secton 17.7 of \cite{G}. Components of
joinability were introduced in \cite{KM}.
\end{rem}

\begin{lemma}\label{LC1} If two proper rays end at the same point of $\del X $ then they are
properly homotopic.
\end{lemma}

\begin{proof} Let the proper rays $c$ and $c'$ end at the point
$p$. AR's are locally simply connected, so we can choose a basic system
of neighborhoods of $p$ in $\widehat X$

\begin{equation*}
\dots \subseteq W_n\subseteq V_n \subseteq U_n=W_{n-1}\subseteq V_{n-1}\subseteq
U_{n-1}\subseteq\dots
\end{equation*}

such that \begin{enumerate}[(i)] \item $[c(n),\infty )\cup [c'(n),
\infty )\subseteq W_n$; \item any two points in $W_n$ can be joined by
a path in $V_n$; \item any loop in $V_n$ is homotopically trivial in
$U_n$.  \end{enumerate} For each $n$ choose a path $\omega _n$ in $V_n$
joining $c(n)$ to $c'(n)$. Choose a trivializing homotopy in $U_{n-1}$
of the loop formed by the segments $[c'(n-1), c'(n)]$,  $[c(n-1), c(n)]$,
$\omega _{n-1}$ and $\omega _n$. Because $\del X$ is a $Z$-set all these
homotopies can be pulled off $\del X$ with the required amount of control. Thus
they can be fitted together to give a proper homotopy between $c$
and $c'$.  
\end{proof}

We denote by $[c]$ the strong shape component of $\del X$ defined by
the proper ray $c$. If some $c'$ in $[c]$ ends at a point of $\del X$
we say that the strong shape component $[c]$ of $\del X$ is {\it non-empty}. In
general there may also be empty strong shape components, meaning examples
of $[c]$ containing no such $c'$. The existence of these is explored
in \cite{GK} where, in particular the following is proved (\cite{GK}
Corollary 3.6 and Theorem 5.1):

\begin{thm}\label{kras} If $\del X$ has an empty strong shape component then it has
uncountably many such, and also uncountably many components of joinability.
\end{thm}

\begin{rem} The proof in \cite{GK} is for compact metrizable spaces in
general, not specifically for boundaries of $CAT(0)$ spaces.  \end{rem}

{\it Proof of Theorem \ref{main}:} The ``only if" direction is
trivial. Assume that any two geodesic rays in $X$ are properly
homotopic. Then $\del X$ has only one component of joinability. By
Theorem \ref{kras}, $\del X$ has no empty strong shape components. So
any strong shape component $[c]$ contains a proper ray $c'$ which ends
at a point $p$ of $\del X$. There is also a geodesic ray ending at $p$,
and, by Lemma \ref{LC1}, it is properly homotopic to  $c'$. It follows
that $\del X$ has only one strong shape component. In other words, $X$
has semistable fundamental group at infinity.  \hfill$\square$

\section{Group Theory}

Let the group $G$ act geometrically (i.e. properly discontinuously and
cocompactly) on the one-ended proper $CAT(0)$ space $X$. Then the discrete
group $G$ is quasi-isometric to $X$, so $G$ has semistable fundamental
group at infinity (in the usual sense) if and only if the same is true of $X$. (See, for example,
Section 18.2 of \cite {G} where quasi-isometry is implicitly proved.) The relevance to group
theory is that in order to show $G$ has semistable fundamental group at
infinity, one need only check the condition on geodesic rays given in
Theorem \ref{main}. And since semistable fundamental group at infinity is a quasi-isometry
invariant, this can be checked on any proper $CAT(0)$ space on which $G$ acts geometrically.

This is of interest because of the following long-studied problem:
\vskip 5pt
{\bf Open Question $2$:} {\it Is it true that every (one-ended) $CAT(0)$ group has semistable
fundamental group at infinity?}
\vskip 5pt 
The theorem of Krasinkiewicz \cite {K} mentioned in Section $1$ implies
that if there exists a $CAT(0)$ group whose boundary does not have
semistable fundamental group at infinity then that boundary is not shape
equivalent to a (connected) locally connected metrizable compact space.
\vskip 5pt

As a corollary to Theorem \ref{main} we have:
\begin{cor} If $\del X$ is path connected then $G$ has semistable fundamental group at infinity.
\end{cor}

Many $CAT(0)$ groups are known to have semistable fundamental group at
infinity, and for some of these groups spaces $X$ exist having path connected boundaries.
All one-ended word-hyperbolic groups
have connected and locally connected boundary in the sense of Gromov; such a boundary is path
connected. So when a one-ended word-hyperbolic group acts geometrically on a proper CAT(0)
space $X$ then $\del X$ is path connected. All Artin groups and Coxeter
groups have semistable fundamental group at infinity, but it has been
conjectured that right angled Artin groups which do not split as direct
products of infinite groups cannot have path connected boundaries.

A well-known example due to Croke and Kleiner \cite{CK} gives a $CAT(0)$ group $G$,
actually a right angled Artin group, which has the following properties:
\begin{enumerate}[(1)]
\item $G$ has semistable fundamental group at infinity;
\item $G$ acts geometrically on two proper $CAT(0)$ spaces whose boundaries are not
homeomorphic;
\item Those boundaries are not path connected.
\end{enumerate}

M. Mihalik has asked:
\vskip 5pt
{\bf Open Question $3$:} {\it If $G$ acts geometrically on $CAT(0)$ spaces $X$ and $X'$ and
if $\del X$ is path connected must $\del X'$ be path connected?}
\vskip 5pt
The matter of when boundaries associated with $CAT(0)$ groups are path connected seems to be
not fully understood.  
\vskip 5pt
{\bf Acknowledgment:} I thank Chris Hruska, Mike Mihalik and Kim Ruane for helpful
comments on an earlier draft of this note.

\bibliographystyle{amsalpha}

\def\cprime{$'$}
\providecommand{\bysame}{\leavevmode\hbox to3em{\hrulefill}\thinspace}
\providecommand{\MR}{\relax\ifhmode\unskip\space\fi MR }
\providecommand{\MRhref}[2]{%
  \href{http://www.ams.org/mathscinet-getitem?mr=#1}{#2}
}
\providecommand{\href}[2]{#2}

\end{document}